\documentclass[11pt,a4paper]{article}
\usepackage[utf8]{inputenc}
\usepackage[T1]{fontenc}
\usepackage{lmodern}

\usepackage[margin=3cm]{geometry}
\usepackage{microtype}
\usepackage{mathrsfs}

\usepackage{todonotes}

\usepackage[english]{babel}

\usepackage{mathtools, amsmath, amssymb, amsthm, slashed}
\numberwithin{equation}{section} 

\theoremstyle{plain} 
\newtheorem{theorem}{Theorem}[section]
\newtheorem{lemma}[theorem]{Lemma}
\newtheorem{proposition}[theorem]{Proposition}
\newtheorem{corollary}[theorem]{Corollary}

\theoremstyle{definition} 
\newtheorem{definition}[theorem]{Definition}

\theoremstyle{remark} 
\newtheorem{remark}[theorem]{Remark}

\usepackage[hidelinks]{hyperref}
\usepackage{tikz}
\usetikzlibrary{arrows.meta,calc,decorations.pathmorphing,positioning}
\usepackage{tikz-cd}

\usepackage{enumitem}
\setlist[enumerate]{label=(\alph*)}
\title{A note on coupled Dirac-harmonic maps}
\author{Balázs Márk Békési}
\date{\today}

\usepackage[
	backend=biber,
	style=alphabetic,
	maxnames=99,
	minnames=99,
	maxcitenames=99,
	mincitenames=99,
	maxalphanames=99,
	minalphanames=99
]{biblatex}

\begin{document}
	\maketitle
	
	\begin{abstract}
		We construct coupled Dirac-harmonic maps from manifolds of the form $S^1\times M$ where $M$ is a spin manifold equipped with a pair of harmonic spinors satisfying a norm condition. This method produces coupled Dirac-harmonic maps from closed manifolds.
	\end{abstract}
	\section{Dirac-harmonic maps}
	Dirac-harmonic maps are the simplest fermionic modification of the Dirichlet energy for maps and were introduced in \cite{diracHarmonic} as an extension of the theory of harmonic maps. Motivated by particle physics a twisted spinor field $\Psi$ is introduced to the Lagrangian
	$$
	L[u,\Psi] = \frac{1}{2}\int_M \|du\|^2 + \langle \Psi, \slashed{D}^u\Psi \rangle \, d\mathrm{Vol}_g
	$$
	where $u:(M,g) \rightarrow (N,h)$ is a smooth map between Riemannian manifolds, $\|du\|^2 = \mathrm{tr}_g(u^*h)$ is the Hilbert-Schmidt norm, $M$ is a spin manifold with spinor bundle $\Sigma M$ equipped with a Hermitian metric $\langle \cdot, \cdot \rangle_{\Sigma M}$ and $\Sigma M \otimes u^*TN$ with the induced metric $\langle \cdot, \cdot \rangle$, $\Psi \in \Gamma(\Sigma M \otimes u^*TN)$ a twisted spinor and the twisted Dirac operator $\slashed{D}^u$ is the composition
	$$
	\Gamma(\Sigma M\otimes u^*TN) \longrightarrow \Gamma(T^*M\otimes \Sigma M\otimes u^*TN) \stackrel{\mathrm{cl}}{\longrightarrow} \Gamma(\Sigma M\otimes u^*TN)
	$$
	where the first arrow denotes the action of $\nabla^{\Sigma M}\otimes u^*\nabla^h$ with $\nabla^{\Sigma M}$ being the spin connection and $\mathrm{cl}$ is the Clifford multiplication. The Euler-Lagrange equation of $L$ is given by
	\begin{align}\label{eq:euler_lagrange}
	\begin{cases}
		\slashed{D}^u \Psi &= 0\\
		\Delta u & = \mathcal{R}(u,\Psi)
	\end{cases}
	\end{align}
	where $\Delta u = \mathrm{tr}_g(\nabla du)$ is the tension field defined with the product connection $\nabla = (\nabla^g)^*\otimes u^*\nabla^h$ built out of the Levi-Civita connections $\nabla^g,\nabla^h$. The curvature term is 
	\begin{align*}
		\mathcal{R}(u,\Psi) &= \frac{1}{2}\mathrm{tr}_g(R^N(\langle \Psi, \mathrm{cl}_{\bullet}(\Psi) \rangle_{\Sigma M}))du(\bullet)\\
		&= \frac{1}{2} \sum_{i=1}^{\dim M} \langle \Psi^\alpha, \mathrm{cl}_{e_i}(\Psi^\beta) \rangle_{\Sigma M} R^N(\theta_\alpha,\theta_\beta)du(e_i) 
	\end{align*}
	where $e_i$ is an orthonormal frame on $TM$, $\Psi = \Psi^\alpha \otimes \theta_\alpha$ locally and the Riemann curvature tensor $R^N$ with the convention $$R^N(X,Y)Z = \nabla^h_X\nabla^h_YZ - \nabla^h_Y\nabla^h_XZ - \nabla^h_{[X,Y]}Z.$$
	\begin{definition}
		A \textbf{Dirac-harmonic map} is a pair $(u,\Psi)$ solving the system (\ref{eq:euler_lagrange}). A Dirac-harmonic map $(u,\Psi)$ is said to be \textbf{uncoupled} if $\mathcal{R}(u,\Psi)\equiv 0$ and \textbf{coupled} otherwise.
	\end{definition}
	
	Obviously if $N$ is flat or $\Psi\equiv 0$ then any Dirac-harmonic map is uncoupled, i.e. $u$ is simply a harmonic map. 
	While there are constructions of coupled Dirac-harmonic maps from non-compact manifolds \cite{diracharmonicexamples}, \cite{ammannexamples} and also geodesics \cite{diracgeo} there is also a substantial amount of uncoupling phenomena known in the closed domain case. For instance in \cite{yanguncoupling} Dirac-harmonic maps between compact surfaces $u:M\rightarrow N$ satisfying the degree-genus condition $g_M=g_N = 0$ or $|g_M - 1|< 2|\deg \,u||g_N - 1|$ are shown to be uncoupled, in \cite{mo2010some} uncoupling is shown from closed domains with positive scalar curvature, in \cite{sunKahler} certain Dirac-harmonic maps between K\"ahler manifolds are shown to uncouple and in \cite{ammanndiracindex} uncoupling is connected to examples from index theory. Together with the kernel jump phenomenon considered in \cite{ammannquestion} 
	the question arises whether there are any coupled Dirac-harmonic maps from closed manifolds.
	\vspace{0.314cm}
	
	In this note, the main result is a product construction in Proposition \ref{prop:product} with coupled periodic Dirac-geodesics. We construct coupled Dirac-geodesics relying on the framework of \cite{diracgeo} and an observation in Lemma \ref{lemma:holonomy} on the holonomy of spinors around periodic Dirac-geodesics. In particular, we can construct coupled Dirac-harmonic maps from compact manifolds and answer the question of \cite{ammannquestion}.
	
	\paragraph{On independent related work of Jost-Zhu}
	After the completion of the note, the author was made aware of the work of Jost-Zhu \cite{JostZhu} where coupled Dirac-harmonic maps from tori $\mathbb{T}^2$ with arbitrary metric and spin structure are obtained with constrained variational methods. The constraint has the same geometric origin as our spinor holonomy condition we exploited for our coupled examples.

	\paragraph{Acknowledgement} The author would like to thank Renan Assimos for introducing him to the question \cite{ammannquestion} in 2024 and Jürgen Jost and Jingyong Zhu for sharing privately their manuscript \cite{JostZhu}. The author's research is funded by the European Research Council (ERC) through StG 101165368 “MAGNETIC”. The use of LLMs reduces to a last (grammar) check at the end of the work. All the mathematical ideas, computations and written text are due to the author.

	\section{Periodic Dirac-geodesics}
	\paragraph{Spinors over $S^1$} The spin structures on $S^1$ are realised by the $\mathrm{Spin}(1) \cong \mathbb{Z}/2\mathbb{Z}\cong\{\pm 1\}$ bundles
	\begin{align*}
		P_0 &= \{(t,k\pi) \in S^1\times S^1 \, \mid \, 0\leq t \leq 2\pi, k=0,1 \}\\
		P_{\frac{1}{2}} &= \{(t,\tfrac{1}{2}t) \in S^1\times S^1\, \mid \, 0\leq t \leq 4\pi\}
	\end{align*}
	where the bundle projections are given by the projection to the first factor and the $\mathrm{Spin}(1)$-action is given by $(a,b).(-1) = (a,b + \pi)$. We refer to the first as spin $0$ structure and to the second as the spin $\tfrac{1}{2}$ structure, in short spin $\nu$ for $\nu \in \{0,\tfrac{1}{2}\}$. The spinor bundles $\Sigma_\nu S^1$ are topologically $S^1\times \mathbb{C}$ but their sections, the spinors, differ in the following way:	
	$$
	\Gamma(\Sigma_\nu S^1)\cong \{ \gamma \in \mathscr{C}^{\infty}([0,2\pi], \mathbb{C}) \mid  \gamma^{(k)}(0)=(-1)^{2\nu}\gamma^{(k)}(2\pi) \text{ for }k \in \mathbb{N}_0\}.
	$$
	Thus for $\nu = 0$ we identify with periodic sections and for $\nu = \tfrac{1}{2}$ with anti-periodic ones. Especially the spin $\tfrac{1}{2}$ case will produce interesting examples.
	
	\paragraph{Dirac-geodesics}
	In \cite{diracgeo} Dirac-harmonic maps from $[0,1]$ into $(N,h)$ are considered for the first time and the system (\ref{eq:euler_lagrange}) is reformulated in Lemma 2.1 of \cite{diracgeo} into
	\begin{align}\label{eq:dirac_geo}
		\nabla_{\dot \gamma}\dot \gamma + R^N(X,Y)\dot \gamma &= 0\\
		\nabla_{\dot\gamma} X&= 0\\
		\nabla_{\dot\gamma} Y&= 0
	\end{align}
	where $\gamma:[0,1] \rightarrow (N,h)$, and $X = \mathrm{Re} \Psi$ and $Y = \mathrm{Im}\Psi$ are interpreted as vector fields along $\gamma$. One can easily prove that $\|\dot\gamma\|,\|X\|,\|Y\|$ and $h(X,Y)$ are constant along the curve $\gamma$. Thus, after a reparametrisation, we can always assume that $\|\dot \gamma\| = 1$.
	
	We are interested in solutions of Equation (\ref{eq:dirac_geo}) where the endpoints of $\gamma$ close up with matching tangents. These are referred to as \textbf{periodic Dirac-geodesics} and $\gamma$ is then interpreted as a map from $S^1$.
	\begin{definition}
		A periodic Dirac-geodesic $(\gamma, \Psi)$ is said to be \textbf{spin $\nu$} if $S^1$ is equipped with the spin $\nu$ structure.
	\end{definition}
	\begin{remark}
		The previous definition uses $H^1(S^1,\mathbb{Z}/2\mathbb{Z})$ for labelling Dirac-geodesics. A similar definition can be made for any spin domain $M$.
	\end{remark}
	
	\begin{lemma}\label{lemma:holonomy}
		Let $(\gamma, \Psi)$ be a spin $\nu$ Dirac-geodesic in $N$. Then the parallel transport satisfies $X(1) = (-1)^{2\nu}X(0)$ and $Y(1) = (-1)^{2\nu}Y(0)$.
	\end{lemma}
	\begin{proof}
		This follows from $X$ and $Y$ being parallel and the (anti)-periodic behaviour observed for the sections of $\Sigma_\nu S^1$.
	\end{proof}
	With all the preparation the roadmap for explicit constructions is clear: We have to solve $\nabla_{\dot \gamma}\dot \gamma + R^N(X,Y)\dot \gamma = 0$ with $S^1$ as a domain and depending on the spin $\nu$ select the solutions where the holonomy matches the one prescribed by the spinor bundle $\Sigma_\nu S^1$. This last condition is a major restriction on the solution space. 
	\subsection{Two-dimensional targets}
	In the following subsection, we restrict our attention to two-dimensional oriented targets $(N,h)$ without boundary. By using the orientation we can define a complex structure $J$ by rotating a tangent vector by $\frac{\pi}{2}$ and denote by $\omega(X,Y) = h(JX,Y)$ its associated symplectic form.
	
	The next proposition shows that the holonomy condition of Lemma \ref{lemma:holonomy} is a "quantisation condition" for the total curvature of fill-ins of Dirac-geodesics and also of the constant $\omega_{\Psi}:=\omega(X,Y)$ associated to the spinor $\Psi$.
	\begin{proposition}\label{prop:quantisation}
		Let $\gamma:S^1\rightarrow N$ be an embedded separating\footnote{i.e. $N \setminus \gamma(S^1)$ is disconnected} coupled Dirac-geodesic and assume that a component $N_\gamma$ of $N \setminus \gamma(S^1)$ has compact closure with boundary $\gamma$. Then
		\begin{enumerate}
			\item $	\frac{1}{2\pi}\int_{N_\gamma} K_N\, dA \in (\nu +\mathbb{Z})$ where $K_N$ denotes the Gaussian curvature.
			\item If $K_N \circ \gamma \equiv \kappa \neq 0$ is constant then $$\omega_{\Psi} = \frac{2\pi (k+\nu)}{\kappa \,L[\gamma]}$$ where $k+\nu = \chi(N_\gamma) - \frac{1}{2\pi}\int_{N_\gamma}K_N\, dA$ where $\chi(N_\gamma)$ is the Euler characteristic of $N_\gamma$ and $L[\gamma]$ the length of $\gamma$.
		\end{enumerate}
	\end{proposition}
	\begin{proof}
		By the unit-speed assumption on $\gamma:[0,\ell]\rightarrow N$ we know that $\dot\gamma, J \dot \gamma$ is an orthonormal frame along $\gamma$. In this frame we see $$ R(X,Y)\dot\gamma = (K_N\circ \gamma)(h(\dot \gamma, Y)X - h(\dot\gamma, X)Y) = - (K_N\circ \gamma)\omega(X,Y)J\dot\gamma = -(K_N\circ \gamma) \omega_{\Psi}J\dot\gamma $$ and combining with (\ref{eq:dirac_geo}) 
		\begin{equation}\label{eq:curv_relation}
		(k_g - (K_N\circ \gamma) \, \omega_\Psi)J\dot \gamma =0
		\end{equation}
		where $k_g = h(\nabla_{\dot \gamma}\dot \gamma, J\dot\gamma)$ is the geodesic curvature. Denote by $\theta$ the angle between $X$ and $\dot\gamma$ and recall the classical relation $\dot \theta = -k_g$. Now (a) follows by an application of Gauss-Bonnet
		$$
		2\pi \chi(N_\gamma) - \int_{N_\gamma}K_N\, dA =  \int_{\gamma}k_g \, dt = - \int_\gamma \dot\theta \, dt = \theta(0)-\theta(\ell)
		$$
		and the fact that $\theta(0)-\theta(\ell)\, \mathrm{mod}\,2\pi$ measures the holonomy around $\gamma$. (b) follows by
		$$
		\omega_\Psi \kappa L[\gamma] = \omega_\Psi\int_{\gamma}(K_N \circ \gamma)\,dt =\int_\gamma k_g \, dt = 2\pi \chi(N_\gamma)- \int_{N_\gamma}K_N\, dA.
		$$
	\end{proof}
		
	\paragraph{Dirac-geodesics in surfaces of revolution} Here, the simple ideas is that on a surface of revolution the parallels are a class of periodic curves for which the Dirac-geodesic equations become easy to solve. Consider the profile curve $\sigma:I \rightarrow \mathbb{R}^2$ defined by $\sigma(s) = (r(s),h(s))$ with $r(s)>0$ satisfying the unit-speed condition $\dot r^2 + \dot h^2= 1$. The surface of revolution $N_\sigma$ with profile curve $\sigma$ is induced by $F:I\times S^1 \rightarrow \mathbb{R}^3$ defined by
	$$
	F(s,\theta) = (r(s)\cos(\theta),r(s)\sin(\theta), h(s))
	$$
	and is equipped with the pullback metric $h = ds^2 + r^2d\theta^2$.
	\begin{proposition}\label{prop:parallels_dirac_geo}
		The separating parallel $t \mapsto (s_0,\tfrac{t}{r(s_0)})$ can be extended to a Dirac-geodesic if $\dot r(s_0) = k+\nu$ with $k+\nu$ determined by Proposition \ref{prop:quantisation} (b).
	\end{proposition}
	\begin{proof}
		Recall that the Gaussian curvature of $N_\sigma$ is $K_N = -\frac{\ddot r}{r}$.
		By Equation \ref{eq:curv_relation} the Dirac-geodesic equations are $\nabla_{\dot \gamma}\dot \gamma - (K_N\circ \gamma) \, \omega_\Psi \, J \dot \gamma = 0 $ which in coordinates $(s,\theta)$ are
		$$
		\begin{cases}
			\ddot \gamma_1 -r\dot r (\dot\gamma_2)^2 - \omega_\Psi \,\ddot r \,\dot\gamma_2 &= 0\\
			\ddot \gamma_2 + 2 \frac{\dot r}{r} \dot \gamma_1\dot \gamma_2 + \omega_\Psi\,\frac{\ddot r}{r^2} \,\dot\gamma_1 &= 0
		\end{cases}
		$$
		
		Now consider a separating parallel $\gamma_{s_0} = (s_0, \frac{1}{r(s_0)}t)$ for which the second ODE vanishes and the first ODE rearranges to
		$$
		\frac{\dot r(s_0)}{r(s_0)} + \omega_\Psi \frac{\ddot r(s_0)}{r(s_0)}=0 \Leftrightarrow -\omega_\psi \ddot r(s_0) = \dot r(s_0)
		$$
		and by $L[\gamma_{s_0}] = 2\pi r(s_0)$ and Proposition \ref{prop:quantisation} (b) the constant becomes $\omega_{\Psi} = -\frac{k+\nu}{\ddot r(s_0)}$ where we used that the Gaussian curvature is the constant $-\frac{\ddot r(s_0)}{r(s_0)}$ along $\gamma_s$. Hence, we obtain
		$$
		k+\nu = \dot r(s_0)
		$$
		where $k + \nu$ is determined by the Euler characteristic and the total Gaussian curvature. Now, we can take any $X_0,Y_0 \in T_{\gamma(0)}N_\sigma$ such that $\omega(X_0,Y_0) = \omega_\Psi$ and parallel translate them to obtain the spinor $\Psi$.
	\end{proof}
	
	\paragraph{Dirac-geodesics in $S^2$} The round sphere of radius $1$ is described by the profile curve $\sigma:(-\frac{\pi}{2},\frac{\pi}{2})\rightarrow \mathbb{R}^2$ with $\sigma(s) = (\cos(s),\sin(s))$. By Proposition \ref{prop:parallels_dirac_geo} the parallels can be extended to Dirac-geodesics if $\dot r(s) = -\sin(s) = k+\nu$. The parallel always bounds a disc $\mathbb{D}_\gamma$ and the Gaussian curvature is $1$ which gives with Proposition \ref{prop:quantisation}
	$$
	k + \nu = \chi(\mathbb{D}_\gamma) - \frac{1}{2\pi} \int_{\mathbb{D}_\gamma}dA = 1 - \frac{1}{2\pi} \mathrm{Area}(\mathbb{D}_\gamma)
	$$
	
	For spin $\nu = 0$ we obtain $\sin(s_0) = 0$, $\mathrm{Area}(\mathbb{D}_\gamma) = 2\pi$ and obtain the equator in this case. This is the case where the Dirac-geodesic is uncoupled as $\omega_\Psi = 0$ in this case.
	
	For spin $\nu = \frac{1}{2}$ we obtain $\sin(s_0) \in \{\pm \frac{1}{2}\}$, $\mathrm{Area}(\mathbb{D}_\gamma)\in \{\pi, 3\pi\}$. This case is coupled and can be extended to a non-trivial Dirac-geodesic. Note that both of the parallels are rotations of each other.
	\paragraph{The hyperbolic plane} We can repeat the construction of surfaces of revolution in the Minkowski space $(\mathbb{R}^3,dx^2+dy^2 - dz^2)$ with profile curve $\sigma(t)=(r(t),h(t))$ satisfying $(\dot r)^2 - (\dot h)^2 = 1$. Then with the same embedding as before, we obtain the same metric $ds^2 + r^2 d\theta^2$ and Proposition \ref{prop:parallels_dirac_geo} is still valid. By choosing $\sigma(s)=(\sinh(s),\cosh(s))$ we obtain the hyperboloid model of the hyperbolic plane and the parallels $t\mapsto (s_0,t)$ can be extended to Dirac-geodesics, provided
	$$
	\dot r(s_0) = \cosh(s_0) = k + \nu = 1 + \frac{1}{2\pi} \mathrm{Area}(\mathbb{D}_\gamma)
	$$ 
	for $k \in \mathbb{N}$. In this case, we obtain coupled periodic Dirac-geodesics for both spin $\nu$ choices.
	
	\begin{remark}\label{rem:compact_target}
		Let $\Gamma<\mathrm{PSL}_2(\mathbb{R})$ be a Fuchsian group such that $\Sigma = \mathbb{H}^2/\Gamma$ is a genus $g$ surface equipped with its hyperbolic metric. If the fundamental domain is large enough and contains a periodic Dirac-geodesic then there is also one in $\Sigma$.
	\end{remark}

	\paragraph{Comparison with Ammann's kernel jump phenomena}
	As a consequence of a variational computation \cite{ammannquestion} derives the following statement.
	\begin{proposition}(Corollary 3 in \cite{ammannquestion})
		Let $(u,\Psi)$ be a Dirac-harmonic map from a closed manifold $(M,g)$. Then either
		\begin{enumerate}
			\item $(u,\Psi)$ is uncoupled or
			\item every open neighbourhood $U\subseteq \mathscr{C}^{\infty}(M,N)$ of $u$ contains an element $u_0\in U$ such that $\dim \ker \slashed D^{u_0} < \dim \ker \slashed D^{u}$.
		\end{enumerate}
	\end{proposition}
	
	In the setting of coupled Dirac-geodesics $(\gamma,\Psi)$ part (b) becomes an immediate observation in dimension two. There are small perturbations of $\gamma$ such that the holonomy condition in Proposition \ref{prop:quantisation} is destroyed and hence the parallel transport does not create spinors that close up according to the choice of the spin structure $\nu$. Thus, there are nearby curves $\gamma_0$ with $\dim \ker \slashed D ^{\gamma_0}=0$.
	
	\section{Some generalisations} 
	There are two immediate methods for obtaining Dirac-geodesics into higher dimensional targets.
	\begin{lemma}
		Let $(\gamma,\Psi)$ be a coupled periodic Dirac-geodesic in $N$ and let $B$ be a Riemannian manifold.
		\begin{enumerate}
			\item Let $\tilde\gamma:S^1 \rightarrow B$ be a periodic geodesic. Then $\gamma\times \tilde\gamma:S^1 \rightarrow N\times B$ with $\Psi \oplus 0$ is a coupled Dirac-geodesic in $N \times B$.
			\item Let $F:N \rightarrow B$ be a totally geodesic\footnote{$\nabla dF = 0$} isometric embedding. Then $(F\circ \gamma, F_*\Psi)$ is a coupled Dirac-geodesic in $B$.
		\end{enumerate}
	\end{lemma} 
	\begin{proof}
		(a) follows by the splitting $\Sigma S^1 \otimes (\gamma\times \tilde\gamma)^*T(N\times B) = \Sigma S^1\otimes \gamma^*TN \oplus \Sigma S^1\otimes \tilde\gamma^*TB$ and the Riemannian curvature splitting as well for the defined spinor. Claim (b) follows by $F_*(R^N(X,Y)\dot\gamma) = R^B(F_*X,F_* Y)F_*\dot \gamma$.
	\end{proof}
	By (b) we obtain examples of Dirac-geodesics in any $S^n$ by embedding $S^2$ in a totally geodesic way and similarly for hyperbolic spaces $\mathbb{H}^n$. With (a) we also know that $S^2 \times S^2$ also contains such coupled Dirac-geodesics.
	
	\paragraph{A product construction} Lastly, we construct coupled Dirac-harmonic maps from manifolds of the form $S^1 \times M$.
	
	\begin{proposition}\label{prop:product}
		Let $(\gamma,\Psi)$ be a coupled periodic Dirac-geodesics in $N$ and $u:M \rightarrow B$ a harmonic map from a spin manifold $M$. Additionally
		\begin{enumerate}
			\item if $\dim M$ is odd assume the existence of harmonic spinors $\psi_1,\psi_2\in \Gamma(\Sigma M)$ such that $|\psi_1|^2-|\psi_2|^2=1$ and
			\item if $\dim M$ is even assume the existence of a harmonic spinor $\psi = (\psi^+,\psi^-)\in \Gamma(\Sigma^+M \oplus \Sigma^-M)$ such that $|\psi^+|^2-|\psi^-|^2=1$
		\end{enumerate} 
		Then there is a coupled Dirac-harmonic map with underlying map $(\gamma,u):S^1 \times M \rightarrow N \times B$.
	\end{proposition}
	\begin{proof}
		Take the induced spin structure on $S^1\times M$. Denote by $S_1 = \mathrm{pr}_1^* \Sigma S^1$ and $S_2 = \mathrm{pr}_2^* \Sigma M$ the pullbacks of the spinor bundles via the canonical projections. By Proposition 5 in \cite{klinker2002spinor} the spinor bundle on the product $S^1\times M$ is
		$$
		S = S'\oplus S'' = \begin{cases}
			S_1 \otimes (S_2^+\oplus S_2^-)& \text{ if } \dim M \text{ is even}\\
			(S_1 \otimes S_2) \oplus (S_1\otimes S_2)& \text{ if } \dim M \text{ is odd}
		\end{cases}
		$$
		where $S_2^+$ and $S_2^-$ in the even case are the even and odd subbundles of $S_2$. The Clifford multiplication along tangent vectors of $S^1$ and $M$ are given by
		$$
		\begin{bmatrix}
			\mathrm{cl}^{S_1}\otimes 1 & 0 \\
			0 & -\mathrm{cl}^{S_1} \otimes 1
		\end{bmatrix} \text{ and }
		\begin{bmatrix}
		0 & 1 \otimes \mathrm{cl}^{S_2} \\
		1 \otimes \mathrm{cl}^{S_2} & 0
		\end{bmatrix}
		$$
		respectively.\footnote{Note, that $S'$ and $S''$ are not the even and odd subbundles of $S$.} This implies that the Dirac operator on $S$ decomposes as well. By
		$$
		S \otimes (\gamma^*TN\oplus u^*TB) \cong \begin{cases}
			S_1 \otimes \gamma^*TN \otimes S_2  \oplus S_1 \otimes S_2\otimes u^*TB & \text{ if } \dim M \text{ is even}\\
			(S_1 \otimes \gamma^*TN \otimes S_2)^{\oplus 2}  \oplus (S_1 \otimes S_2\otimes u^*TB)^{\oplus 2}& \text{ if } \dim M \text{ is odd}
		\end{cases}
		$$
		the twisted spinor is defined as
		$$
		\hat\Psi = \begin{cases}
			\Psi\otimes \psi^+ \oplus \Psi \otimes \psi^- \oplus 0\oplus 0& \text{ if } \dim M \text{ is even}\\
			\Psi\otimes \psi_1 \oplus \Psi\otimes \psi_2 \oplus0 \oplus0& \text{ if } \dim M \text{ is odd}
		\end{cases}.
		$$
		By construction $\slashed D^{(\gamma,u)}\hat \Psi=0$. Let $e_1$ be a unit tangent vector along the $S^1$-factor and $\Psi = \Psi^\alpha\otimes \theta_\alpha$ a local expression for $\Psi$. The curvature term in the odd case is
		\begin{align*}
			\mathcal{R}^{N\times B}((\gamma,u),\hat \Psi) &= \frac{1}{2} \left \langle \Psi^\alpha\otimes \begin{bmatrix}
				\psi_1 \\ \psi_2
			\end{bmatrix} , \mathrm{cl}^{S}_{e_1}\left(\Psi^\beta\otimes \begin{bmatrix}
			\psi_1 \\ \psi_2
			\end{bmatrix} \right) \right \rangle_S R^{N\times B}(\theta_\alpha,\theta_\beta) d(\gamma,u)(e_i)\\
			&= \frac{1}{2} \left \langle \Psi^\alpha, \mathrm{cl}^{S_1}_{e_1}(\Psi^\beta) \right \rangle_{S_1} ( \langle \psi_1,\psi_1 \rangle_{S_2}-\langle \psi_2,\psi_2 \rangle_{S_2}) R^{N}(\theta_\alpha,\theta_\beta) \dot\gamma\\
			& = \mathcal{R}^N(\gamma,\Psi) \oplus 0
		\end{align*}
		and similarly in the even case. By $\Delta (\gamma,u) = (\nabla_{\dot\gamma}\dot \gamma, \Delta u)$ the assertion follows.
	\end{proof}
	
	\begin{remark}
		The proof can also be modified to products involving two coupled spin $0$ periodic Dirac-geodesics $(\gamma_i,\Psi_i)$ $i=1,2$ by defining 
		$$\hat\Psi =\Psi_1\otimes \begin{bmatrix}
			\psi_1 \\ \psi_2
		\end{bmatrix} \oplus  \begin{bmatrix}
		\phi_1 \\ \phi_2
		\end{bmatrix} \otimes \Psi_2.$$
		Here the spin $0$ is needed for the existence of the parallel spinors $\psi_1,\psi_2,\phi_1,\phi_2$ satisfying $|\psi_1|^2-|\psi_2|^2=1$ and $\langle \phi_1,\phi_2 \rangle+ \langle \phi_2,\phi_1 \rangle = 1$.
	\end{remark}
	
	\begin{corollary}
		There is a coupled Dirac-harmonic map from flat tori $\mathbb{T}^n$. The domain and target can be chosen to be compact K\"ahler.
	\end{corollary}
	\begin{proof}
		The harmonic spinor assumption of Proposition \ref{prop:product} are satisfied on flat tori $\mathbb{T}^n$ by them admitting parallel spinors. Now take an example of a coupled Dirac-geodesic constructed in Section 2 and apply Proposition \ref{prop:product} to the identity map $\mathrm{Id}_{\mathbb{T}^n}$. The K\"ahler case follows by taking a coupled Dirac-geodesic into $\mathbb{CP}^1\cong S^2$ and $u:\mathbb{T}^{2n-1}\rightarrow \mathbb{T}^{2n}$ a totally geodesic embedding.
	\end{proof}
	
	\paragraph{On the spinor condition}
	By applying the Lichnerowicz formula $\slashed D^2\psi = \nabla^*\nabla\psi + \frac{1}{4}\mathrm{Scal}_g\, \psi$ to each harmonic spinor $\psi_1,\psi_2$ and applying the condition $|\psi_1|^2 - |\psi_2|^2 = 1$ gives the scalar curvature condition
	$$
	\mathrm{Scal}_g = 4 (|\nabla \psi_2|^2 -|\nabla \psi_1|^2).
	$$
	If $\mathrm{Scal}_g = 0$ then we obtain $|\nabla \psi_1|^2 = |\nabla\psi_2|^2$. There is also an interesting subcase of $\psi_2\equiv 0$ which also contains parallel spinors.
	
	\newpage
	\printbibliography
	
	\textbf{Università Commerciale Luigi Bocconi}\\
	\href{mailto:balazs.bekesi@unibocconi.it}{balazs.bekesi@unibocconi.it}
\end{document}